\documentclass[11pt]{article}
\usepackage{amsmath, amssymb, amsthm, mathtools}
\usepackage{geometry}
\usepackage[colorlinks=true,linkcolor=blue,citecolor=blue,urlcolor=blue]{hyperref}
\usepackage{graphicx}

\newtheorem{theorem}{Theorem}[section]
\newtheorem{lemma}[theorem]{Lemma}
\newtheorem{proposition}[theorem]{Proposition}

\newtheorem{definition}[theorem]{Definition}
\newtheorem{remark}[theorem]{Remark}

\DeclareMathOperator{\Arg}{Arg}

\DeclareMathOperator{\cotan}{cot}

\title{On the Spectral Region of $4$-Cycle Stochastic Matrices}
\author{Brando Vagenende, Brecht Verbeken, Andres Algaba, Marie-Anne Guerry \footnote{Department Business Technology and Operations, Data Analytics Laboratory, Vrije Universiteit Brussel (VUB), Pleinlaan 2, Brussels, 1050, Belgium (brando.vagenende@vub.be,
brecht.verbeken@vub.be, andres.algaba@vub.be, marie-anne.guerry@vub.be).}}
\date{}
\begin{document}

\maketitle

\begin{abstract}
We study the spectrum of $4$-cycle row-stochastic matrices.
For real eigenvalues the spectral region is \([-1,1]\). For nonreal eigenvalues $\lambda=a+ib$ we derive necessary conditions in terms of the real and imaginary parts, including the inequality $a+|b|\le 1$ and the condition $(b^2+a^2+a)^2+2a^2-b^2\ge 0$.
We also prove conversely that every point in the corresponding interior region occurs as an eigenvalue of a $4$-cycle matrix. The proof is organized through a reformulation of the characteristic equation, an argument parametrization, a convex-analytic criterion, and explicit boundary constructions. Hence, the spectral region for the $4$-cycle row-stochastic matrices is exactly and explicitly determined.
\end{abstract}

{\textbf{Keywords:}} Nonnegative matrices, stochastic matrices, cycle stochastic matrices, eigenvalue region. 

{\textbf{AMS subject classifications:}} 15A18, 15B51

\section{Introduction}

A stochastic matrix is a matrix with nonnegative entries whose row sums are one. Such matrices are fundamental in the study of finite-state Markov chains, where they represent the transition probabilities between states of a stochastic process \cite{Bartholomew}. The eigenvalues of a stochastic matrix contain important information about the dynamics of the underlying process, including its long-term behaviour, stationary distributions, and convergence rates \cite{pillai2005perron, delbianco2023markov, racoceanu1995new, meyer2000applied}.

A classical problem in matrix theory is to describe the set of all possible eigenvalues of stochastic matrices. The general answer is given by the Karpelevich theorem~\cite{Karpelevich} and a lot of extended research is known \cite{munger2024demystifying, kirkland2022stochastic}. In contrast, the spectral properties of subclasses of stochastic matrices remain far less understood, leaving substantial scope for further investigation.

Doubly stochastic matrices, which form a distinguished subclass of stochastic matrices characterized by unit row and column sums, have been studied extensively from a spectral perspective. Although a number of important results concerning their eigenvalue regions have been obtained, many questions remain open. The literature contains several conjectures and partial results aimed at describing these regions more completely \cite{mashreghi2007conjecture, kim2022conjectures}. Similar investigations have been carried out for other classes of structured matrices, such as monotone stochastic matrices \cite{VagenendeVerbekenGuerry2026}, Metzler matrices \cite{domka2022spectrum} and Leslie matrices \cite{kirkland1992eigenvalue}.

In this paper, we consider another stochastic subset, namely the \(n\)-cycle stochastic matrices. Such matrices are row-stochastic and have the combinatorial structure of a directed $n$-cycle with self-loops. Research about this type of graphs can be found in e.g. \cite{gerbner2018number}. Especially, this paper focuses on the $4$-cycle matrices considered by Ran and Teng~\cite{ran2024nonnegative}. In that paper, the authors obtained complete results for the $3\times 3$ case and formulated, for the $4$-cycle stochastic matrices, a conjectural description of the nonreal spectral region. The present paper proves that conjecture and extends it.

Our proof is inspired by the trigonometric method introduced by Dmitriev and Dynkin~\cite{Dynkin} for the low-dimensional Karpelevich regions $\Theta_n$. The method starts from a multiplicative constraint obtained from the characteristic equation, reparametrizes this constraint using arguments of complex numbers, and then reduces the spectral question to a trigonometric optimization problem governed by convexity and majorization. For the present matrices, this approach leads to a particularly transparent organization of the proof.

The proof conists of three main parts. First, the characteristic equation is rewritten as an exact multiplicative relation. Second, this relation is converted into an equivalent equation involving arguments and logarithms of moduli, which yields a convex-analytic existence criterion. Third, that criterion is analyzed to obtain the necessary inequalities, the realizing matrices for the boundaries, and the converse construction for every point in the strict interior of the spectral region.

For transparency, we also record that the proof was developed with partial assistance from a large language model. A separate paper documents that process, including the iterative generate-referee-repair workflow and the role of human verification in closing the correctness-critical steps~\cite{verbeken2026early}. We mention this here only as a disclosure. The purpose of the present paper is the mathematical result itself and a self-contained proof.

\section{Background and context}

Dmitriev and Dynkin~\cite{Dynkin} characterized the Karpelevich regions $\Theta_n$ for small values of $n$ by reformulating the characteristic equation in \(\lambda\) by introducing an argument parametrization $u_i=\Arg(\lambda - \alpha_i)$ for \(i \in \{1, \ldots, n\}\). This transforms the problem of finding the boundary of the eigenvalue region into a trigonometric optimization problem, where convexity and majorization can be used. The same pattern underlies our proof.

Ran and Teng~\cite{ran2024nonnegative} studied spectral regions for row-stochastic matrices with prescribed zero patterns. One can verify that for $\alpha_1,\alpha_2,\alpha_3\in[0,1)$, the $3$-cycle stochastic matrices
\[
\begin{pmatrix}
\alpha_1 & 1-\alpha_1 & 0 \\
0 & \alpha_2 & 1-\alpha_2 \\
1-\alpha_3 & 0 & \alpha_3\\
\end{pmatrix}
\]
have the same spectral region as the set of matrices
\[
\begin{pmatrix}
\alpha_1 & 0 & 1-\alpha_1\\
1-\alpha_2 & \alpha_2 & 0 \\
0 & 1-\alpha_3 & \alpha_3\\
\end{pmatrix}.
\]
The spectral region of the latest set of matrices is fully determined in \cite[Proposition 10]{ran2024nonnegative}.

Further, for the $4$-cycle matrices as in Definition \ref{def:4-cycle}, Ran and Teng conjectured that the nonreal spectral region is determined by the Karpelevich constraint $x+|y|\le 1$ together with the algebraic boundary
\[
(y^2+x^2+x)^2+2x^2-y^2=0.
\]
We prove in Theorem~\ref{thm:main} this conjecture of Ran en Teng.

\section{The $4$-cycle matrices and the spectral region}

\begin{definition}\label{def:4-cycle}
For $\alpha_1,\alpha_2,\alpha_3,\alpha_4\in[0,1)$, the $4$-cycle stochastic matrix is defined as
\[
A(\alpha_1,\alpha_2,\alpha_3,\alpha_4)=
\begin{pmatrix}
\alpha_1 & 1-\alpha_1 & 0 & 0\\
0 & \alpha_2 & 1-\alpha_2 & 0\\
0 & 0 & \alpha_3 & 1-\alpha_3\\
1-\alpha_4 & 0 & 0 & \alpha_4
\end{pmatrix}.
\]
\end{definition}
\begin{theorem}[Main theorem]\label{MainTheorem}
For the $4$-cycle stochastic matrices the spectral region is
\[\{a+ib|0\le a < 1, a+|b|\le 1, G(a,b) \geq 0\} \cup [-1, 1].\]
\end{theorem}
The spectral region specified in Theorem \ref{MainTheorem} is visualized in Figure \ref{spectral_region}.

Firstly, we study the real spectral region in Theorem \ref{prop:real-eigenvalues}, followed by the nonreal spectral region in Theorem \ref{thm:main}.

\begin{theorem}[Main theorem, real case]\label{prop:real-eigenvalues}
The real spectral region of the $4$-cycle stochastic matrices is the interval $[-1,1]$.
\end{theorem}

\begin{proof}
Let $\lambda$ be a real eigenvalue of
$A(\alpha_1,\alpha_2,\alpha_3,\alpha_4)$. Since $A$ is row-stochastic and
nonnegative, its spectral radius is at most $1$. Hence
\[
|\lambda|\le 1,
\]
so every real eigenvalue lies in $[-1,1]$.

Conversely, since \(r=1\) is always an eigenvalue of a stochastic matrix, let $r\in[-1,1)$. Choose
\[
x=\frac{1-r}{2}\in(0,1],
\]
and set
\[
\alpha_1=\alpha_2=\alpha_3=\alpha_4=1-x.
\]
Then
\[
A=
\begin{pmatrix}
1-x & x & 0 & 0\\
0 & 1-x & x & 0\\
0 & 0 & 1-x & x\\
x & 0 & 0 & 1-x
\end{pmatrix}
=(1-x)I+xP,
\]
where
\[
P=
\begin{pmatrix}
0&1&0&0\\
0&0&1&0\\
0&0&0&1\\
1&0&0&0
\end{pmatrix}.
\]
The eigenvalues of $P$ are
\[
1,\quad i,\quad -1,\quad -i.
\]
Therefore the eigenvalues of $A$ are
\[
1,\qquad 1-x+ix,\qquad 1-2x,\qquad 1-x-ix.
\]
In particular,
\[
1-2x=r.
\]
Thus every $r\in[-1,1)$ occurs as a real eigenvalue.
\end{proof}

Secondly, for nonreal eigenvalues, we now state Theorem \ref{thm:main}, which resolves Conjecture~20 of Ran and Teng~\cite{ran2024nonnegative} for the $4$-cycle stochastic matrices. Both the real and nonreal case are visualized in Figure \ref{spectral_region}.

\begin{figure}[htbp]
\centering
\includegraphics[width=0.5\textwidth]{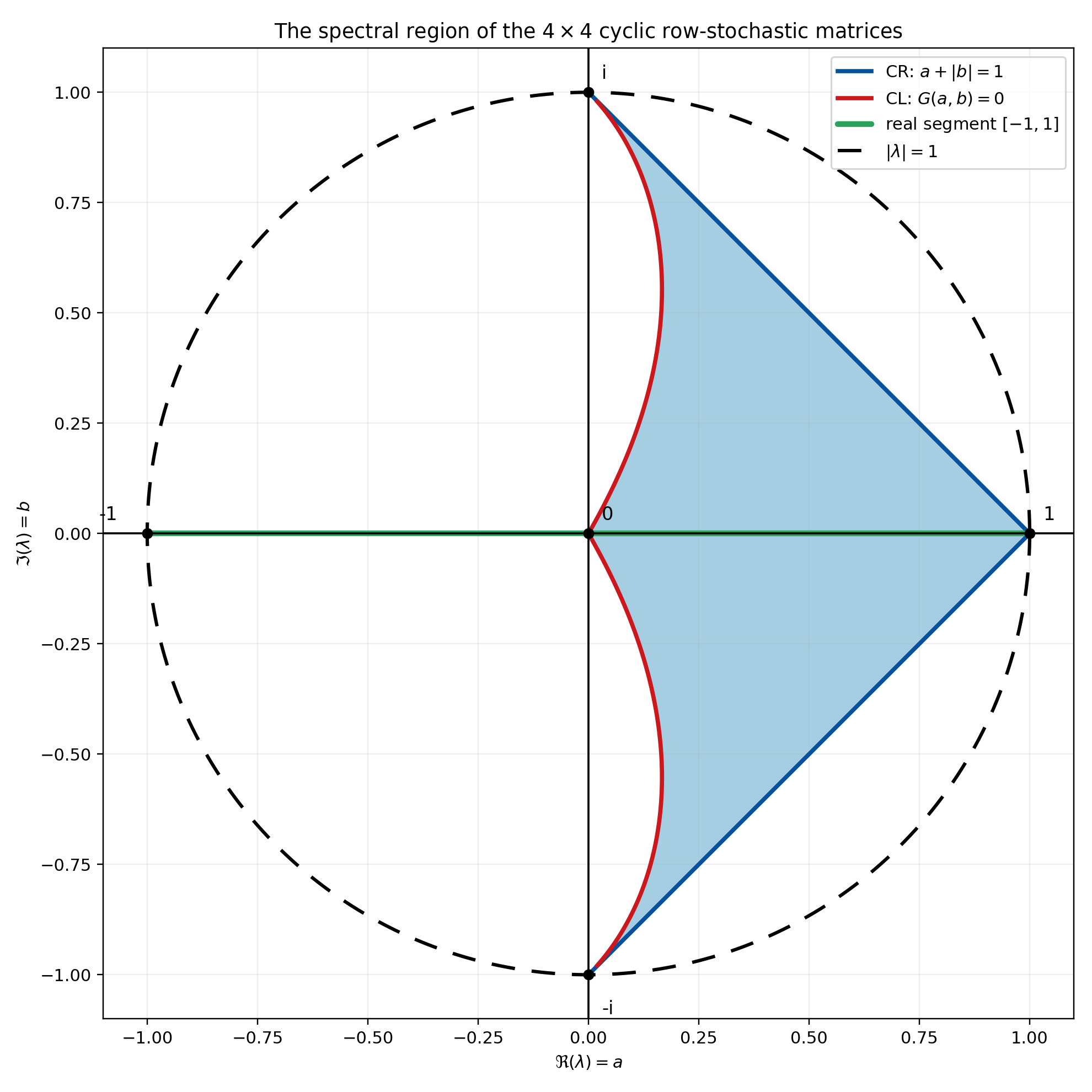}
\caption{The spectral region of the \(4\)-cycle stochastic matrices, where $G(a,b)=(b^2+a^2+a)^2+2a^2-b^2$.}
\label{spectral_region}
\end{figure}

\begin{definition}\label{def:G}
We define the function
\begin{equation}\label{Eq:G(a,b)}
G(a,b):=(b^2+a^2+a)^2+2a^2-b^2.
\end{equation}
\end{definition}
and introduce the set
\begin{equation}\label{spectral region as set}
\mathcal{R}:=\{a+ib|0\le a < 1, a+|b|\le 1, G(a,b) \geq 0\}
\end{equation}
In what follows, we prove that $\mathcal{R}$ is the nonreal spectral region for the $4$-cycle matrices.

\begin{theorem}[Main theorem, nonreal case]\label{thm:main}
Consider the $4$-cycle stochastic matrices
\[
A(\alpha_1,\alpha_2,\alpha_3,\alpha_4)=
\begin{pmatrix}
\alpha_1 & 1-\alpha_1 & 0 & 0\\
0 & \alpha_2 & 1-\alpha_2 & 0\\
0 & 0 & \alpha_3 & 1-\alpha_3\\
1-\alpha_4 & 0 & 0 & \alpha_4
\end{pmatrix},
\qquad \alpha_1,\alpha_2,\alpha_3,\alpha_4\in[0,1).
\]
Let $\lambda=a+ib\in\sigma(A(\alpha_1,\alpha_2,\alpha_3,\alpha_4))$ with $b\neq 0$.
Then:
\begin{enumerate}
\item $0\le a < 1$.
\item $a+|b|\le 1$.
\item $G(a,b) \geq 0$

\end{enumerate}
Conversely, if $b>0$ and
\[
0\le a < 1,\qquad a+b<1,\qquad G(a,b)>0,
\]
then there exist $\alpha_1,\alpha_2,\alpha_3,\alpha_4\in[0,1)$ such that $\lambda=a+ib\in\sigma(A(\alpha_1,\alpha_2,\alpha_3,\alpha_4))$.
The lower half-plane follows by conjugation.

Moreover:
\begin{itemize}
\item The segment $CR:\ \lambda=1-x+ix\ (x\in[0,1])$ is realized by the matrix with $\alpha_1=\alpha_2=\alpha_3=\alpha_4=1-x$.
\item The curve portion $CL$ in the upper half-plane joining $i$ to $0$ defined by $G(a,b)=0$
is realized (except for \(0\) which is a limit point) by
\[
A_L(\alpha)=
\begin{pmatrix}
\alpha & 1-\alpha & 0 & 0\\
0 & 0 & 1 & 0\\
0 & 0 & 0 & 1\\
1 & 0 & 0 & 0
\end{pmatrix},
\qquad \alpha\in[0,1).
\]
\item Every point strictly inside $\{b>0,\ 0\le a < 1,\ a+b<1,\ G(a,b)>0\}$ is also realized as an eigenvalue of a $4$-cycle stochastic matrix.
\end{itemize}
\end{theorem}

Because all matrices under consideration are real, $\lambda\in\sigma(A)$ implies $\overline{\lambda}\in\sigma(A)$, where \(\sigma(A)\) stands for the spectrum of \(A\). Accordingly, in the remainder of the paper we may assume $b:=\operatorname{Im} \lambda>0$ and restore $|b|$ at the end. So, we restrict ourselves to \(\sigma(A)\cap H\), where \(
H:=\{z\in\mathbb C:\operatorname{Im}z>0\}
\).

The remainder of the paper is devoted to proving Theorem~\ref{thm:main}. In particular, the theorem confirms the conjectured left boundary $G(a,b)=0$ from~\cite{ran2024nonnegative} and shows that, together with the right Karpelevich boundary, it yields the full nonreal spectral region for these matrices. For the convenience of the reader, we summarize the logical structure.

\begin{remark}[Road map]\label{rem:roadmap}
The proof of Theorem~\ref{thm:main} proceeds as follows.
\begin{enumerate}
\item In Section~\ref{sec:eig-reformulation}, the characteristic equation is rewritten as a multiplicative constraint.\newline It is then parametrized in terms of arguments $u_i=\Arg(\lambda - \alpha_i)$, for \(i \in \{1, \ldots, n\}\), and a scalar function $F$, defined in (\ref{functionF}). This part is inspired by the trigonometric method introduced by Dmitriev and Dynkin~\cite{Dynkin} for the low-dimensional Karpelevich regions $\Theta_n$.
\item In Section~\ref{sec:criterion}, this leads to the exact existence criterion, namely Theorem \ref{prop:criterion}: a nonreal number $\lambda$ occurs as an eigenvalue if and only if $0$ belongs to the image of the continuous function $\Psi$ on the feasible set $P$, which are introduced in Definition \ref{def:P-Psi}.
\item In Section~\ref{sec:necessary}, the criterion is analyzed to obtain the necessary conditions $0\le a < 1$, $a+b\le 1$, and $G(a,b)\ge 0$. These results are stated in Lemma \ref{lem:a-range}, Proposition \ref{prop:right-boundary-necessary} and Proposition \ref{prop:G-condition}.
\item In Section~\ref{sec:boundary}, Proposition \ref{prop:CR} and Propsition \ref{prop:CL} give two explicit realizing matrices for respectively the boundary pieces $CR$ and $CL$.
\item In Section~\ref{sec:converse}, Lemma \ref{lem:starconvex} and Lemma \ref{Geometric intersection with the left boundary} yield the converse that every strict interior point of the eigenvalue region \(\mathcal{R}\) is realized.
\end{enumerate}
\end{remark}

\section{Reformulation of the characteristic equation}\label{sec:eig-reformulation}





\begin{proposition}[Multiplicative eigenvalue constraint]\label{prop:multconstraint}
Let $\lambda\notin\mathbb{R}$. Then $\lambda\in\sigma(A(\alpha_1,\alpha_2,\alpha_3,\alpha_4))$ if and only if
\begin{equation}\label{eq:multconstraint-new}
(\lambda-\alpha_1)(\lambda-\alpha_2)(\lambda-\alpha_3)(\lambda-\alpha_4)=(1-\alpha_1)(1-\alpha_2)(1-\alpha_3)(1-\alpha_4).
\end{equation}
\end{proposition}

\begin{proof}

$\lambda\in\sigma(A(\alpha_1,\alpha_2,\alpha_3,\alpha_4))$ is equivalent with $\lambda$ is a solution of the characteristic equation of the matrix $A(\alpha_1,\alpha_2,\alpha_3,\alpha_4)$, which can be expressed as  (\ref{eq:multconstraint-new}).
\end{proof}

Introducing for $k\in\{1,\dots,4\}$ the notations
\begin{equation}\label{def:tk-z}
t_k=1-\alpha_k \,\,\, \,\,\, \text{and}\,\,\, z=\lambda-1
\end{equation}
let us rewrite the characteristic equation \eqref{eq:multconstraint-new} equivalent to
\begin{equation}\label{eq:zproduct-new}
(z+t_1)(z+t_2)(z+t_3)(z+t_4)=t_1t_2t_3t_4.
\end{equation}
where all $t_k \in (0,1]$.

Since $\lambda=a+ib$, we have
\[
z=x+iy,\qquad \,\,\, \text{where} \,\,\,  x=a-1 \,\,\, \text{and} \,\,\, y=b.
\]
In the necessity part we assume $y=b>0$.

\begin{definition}\label{def:uF}
For $t>0$, define $u(t)=\Arg(z+t)\in(0,\pi)$, and set
\[
m:=\Arg(z+1)=\Arg(\lambda),\qquad M:=\Arg(z)=\Arg(\lambda-1).
\]
Remark that 
\[
u(t)=\arctan\!\Big(\frac{y}{x+t}\Big)
\]
and that for $u\in[m,M)$ this relation can be expressed as
\[
t(u)=y\cotan u-x.\]
Further, let us introduce the function
\begin{equation}\label{functionF}
    F(u)=\log|z+t(u)|-\log t(u).
\end{equation}
\end{definition}

\begin{lemma}\label{lem:parametrization}
The following properties hold:
\begin{enumerate}
\item For each $u\in[m,M)$,
\[
|z+t(u)|=y\csc u.
\]
\item The map $t\mapsto u(t)$ is continuous and strictly decreasing on $(0,1]$.
\item On $(0,1]$ the map $t\mapsto u(t)$ is a bijection onto $[m,M)$.
\end{enumerate}
\end{lemma}

\begin{proof}
Due to symmetric reasons we can assume that $y>0$. Therefore, $z+t=(x+t)+iy$ lies in the upper half-plane and the relation $u(t)=\Arg(z+t)$ is equivalent to
$t(u)=y\cotan u-x$. Moreover,
\[
|z+t|^2=(x+t)^2+y^2=y^2(\cotan^2 u+1)=y^2\csc^2 u,
\]
so $|z+t(u)|=y\csc u$.

Since
\[
u(t)=\arctan\!\Big(\frac{y}{x+t}\Big)
\]
 $u(t)$ is continuous and strictly decreasing in $t>0$. Therefore, restricted to $t\in(0,1]$, it is a continuous decreasing bijection onto the interval with endpoints $u(1)=m$ and $\lim_{t\downarrow 0}u(t)=M$.
\end{proof}

\section{Alternative eigenvalue criterion}\label{sec:criterion}

\begin{lemma}[Exact equivalence]\label{lem:equiv}
A nonreal number $\lambda$ satisfies \eqref{eq:zproduct-new} for some $t_1,t_2,t_3,t_4\in(0,1]$
if and only if there exist $u_1,u_2,u_3,u_4\in[m,M)$ such that
\begin{align}
 u_1+u_2+u_3+u_4&=2\pi,\label{eq:sumargs-new}\\
 F(u_1)+F(u_2)+F(u_3)+F(u_4)&=0.\label{eq:sumF-new}
\end{align}
\end{lemma}

\begin{proof}
Assume first that \eqref{eq:zproduct-new} holds with $t_k\in(0,1]$, and set $u_k=\Arg(z+t_k)$. By Lemma~\ref{lem:parametrization}, $u_k\in[m,M)$. Writing $z+t_k=|z+t_k|e^{iu_k}$ and comparing arguments in \eqref{eq:zproduct-new}, we find
\[
u_1+u_2+u_3+u_4\equiv 0\pmod{2\pi}.
\]
Since each $u_k\in(0,\pi)$, the sum lies in $(0,4\pi)$, so necessarily $u_1+u_2+u_3+u_4=2\pi$. Comparing moduli in \eqref{eq:zproduct-new} gives
\[
\prod_{k=1}^4|z+t_k|=\prod_{k=1}^4 t_k.
\]
Taking logarithms and using the definition of $F$ results in \eqref{eq:sumF-new}.

Conversely, assume that $u_1,\dots,u_4\in[m,M)$ satisfy \eqref{eq:sumargs-new} and \eqref{eq:sumF-new}. Define $t_k=t(u_k)$. Then $t_k\in(0,1]$ by Lemma~\ref{lem:parametrization}. Also,
\[
\prod_{k=1}^4(z+t_k)=\Big(\prod_{k=1}^4|z+t_k|\Big)e^{i(u_1+u_2+u_3+u_4)}=\prod_{k=1}^4|z+t_k|.
\]
Because \eqref{eq:sumF-new} is satisfied, we have $\prod_{k=1}^4 |z+t_k|=\prod_{k=1}^4 t_k$ and obtain \eqref{eq:zproduct-new}.
\end{proof}

\begin{definition}\label{def:P-Psi}
Define the set
\[
P:=\Big\{(u_1,u_2,u_3,u_4)\in[m,M)^4:\ u_1+u_2+u_3+u_4=2\pi\Big\}
\]
and the function
\[
\Psi(u_1,u_2,u_3,u_4):=\sum_{k=1}^4 F(u_k).
\]
\end{definition}

\begin{proposition}[Criterion in terms of $P$ and $\Psi$]\label{prop:criterion}
For nonreal $\lambda$,
\begin{equation}\label{eq:criterion-new}
\lambda\in\sigma(A(\alpha_1,\alpha_2,\alpha_3,\alpha_4))\ \text{for some }\alpha_1,\alpha_2,\alpha_3,\alpha_4\in[0,1)
\iff P\neq\emptyset\ \text{and}\ 0\in\Psi(P).
\end{equation}
Moreover, $P$ is convex and hence connected, and $\Psi$ is continuous on $P$.
\end{proposition}

\begin{proof}
Equivalence (\ref{eq:criterion-new}) is exactly the result of Lemma~\ref{lem:equiv} translated into the notation of Definition~\ref{def:P-Psi}. The set $P$ is the intersection of the box $[m,M)^4$ with the affine hyperplane $u_1+u_2+u_3+u_4=2\pi$, so is convex. Continuity of $\Psi$ follows from continuity of $F$.
\end{proof}

\section{Necessary eigenvalue conditions}\label{sec:necessary}

This section derives the three necessary conditions in Theorem~\ref{thm:main} from Proposition~\ref{prop:criterion}.

\begin{lemma}\label{lem:a-range}
If $\lambda=a+ib$ is a nonreal eigenvalue of some matrix $A(\alpha_1,\alpha_2,\alpha_3,\alpha_4)$, then $0\le a<1$.
\end{lemma}

\begin{proof}
Assume $\lambda=a+ib\in\sigma(A(\alpha_1,\alpha_2,\alpha_3,\alpha_4))$ and $b>0$. By Proposition~\ref{prop:criterion}, the set $P$ is nonempty. For $(u_1,u_2,u_3,u_4)\in P$, we have
\[4m \leq u_1+u_2+u_3+u_4 = 2\pi \leq 4M.
\]
since $u_k\in[m,M)$. Hence, the interval $[m,M)$ must contain $\pi/2$:
\begin{equation}\label{eq:mMpi2-new}
m\le \frac{\pi}{2}<M.
\end{equation}
Since $m=\Arg(\lambda)=\Arg(a+ib)$ with $b>0$, the condition $\Arg(a+ib)\le \pi/2$ is equivalent to $a\ge 0$. Likewise, $M=\Arg(\lambda-1)=\Arg((a-1)+ib>\pi/2$ is equivalent to $a-1<0$. Hence $0\le a<1$.
\end{proof}

\begin{lemma}\label{lem:F-convex}
The function $F$ is strictly convex on $(0,\pi)$.
\end{lemma}

\begin{proof}
From Definition~\ref{def:uF} and Lemma \ref{lem:parametrization},
\begin{equation}\label{F(u) goniometrical form}
F(u)=\log(y\csc u)-\log(y\cotan u-x),\qquad y>0.
\end{equation}
Differentiating gives
\[
F''(u)=\frac{x^2+y^2}{(y\cotan u-x)^2\sin^2 u}.
\]

Since $y>0$, $F''(u)>0$ on $(0,\pi)$.
\end{proof}

\begin{proposition}\label{prop:right-boundary-necessary}
If $\lambda=a+ib$ is a nonreal eigenvalue of some matrix $A(\alpha_1,\alpha_2,\alpha_3,\alpha_4)$, then
\[
a+b\le 1.
\]
\end{proposition}

\begin{proof}
Again assume $b>0$. For all $(u_1,u_2,u_3,u_4)\in P$, we have by Lemma~\ref{lem:F-convex} and the Jensen' inequality,
\[
\frac{1}{4}\sum_{k=1}^4 F(u_k)\ge F\!\Big(\frac{u_1+u_2+u_3+u_4}{4}\Big)=F(\pi/2).
\]
Hence
\begin{equation}\label{eq:jensen-new}
\Psi(u_1,u_2,u_3,u_4)\ge 4F(\pi/2)
\qquad\text{for all }(u_1,u_2,u_3,u_4)\in P.
\end{equation}
If $\lambda \in \sigma(A(\alpha_1,\alpha_2,\alpha_3,\alpha_4))$, then $0\in\Psi(P)$ by Proposition~\ref{prop:criterion}, so \eqref{eq:jensen-new} forces $F(\pi/2)\le 0$.

Now
\[
F(\pi/2)=\log(y\csc(\pi/2)) - \log(y\cotan(\pi/2)-x)=\log(y)-\log(-x).
\]
Since $y=b$ and $x=a-1$, this becomes
\[
F(\pi/2)=\log\!\Big(\frac{b}{1-a}\Big).
\]
Thus $F(\pi/2)\le 0$ if and only if $b\le 1-a$, that is, $a+b\le 1$. Replacing $b$ by $|b|$ gives the result for arbitrary nonreal $\lambda$.
\end{proof}

Due to Proposition \ref{prop:criterion}, we may assume that \(P \neq \emptyset\) in the following lemmas.

\begin{lemma}[Unbounded-supremum regime]\label{lem:unbounded}
If \(P \neq \emptyset\) and $3m+M\le 2\pi$, then $\sup_P\Psi=+\infty$.
\end{lemma}

\begin{proof}
Fix $u_4\in[m,M)$ and define $u_1=u_2=u_3=(2\pi-u_4)/3$. Then $u_1+u_2+u_3+u_4=2\pi$. If $u_4$ is sufficiently close to $M$ from below, the hypothesis $3m+M\le 2\pi$ implies $u_1\in[m,M)$, hence $(u_1,u_2,u_3,u_4)\in P$.

As $u_4\uparrow M$, we have $t(u_4)\downarrow 0$, so
\[
F(u_4)=\log|z+t(u_4)|-\log t(u_4)\to +\infty.
\]
Meanwhile $u_1=(2\pi-u_4)/3$ converges to $(2\pi-M)/3\in[m,M)$, so $F(u_1)$ is bounded. Therefore
\[
\Psi(u_1,u_2,u_3,u_4)=3F(u_1)+F(u_4)\to +\infty.
\]
\end{proof}

\begin{lemma}[Tight regime]\label{lem:tight}
Assume \(P \neq \emptyset\) and $3m+M>2\pi$, and define $U:=2\pi-3m$. Then $U\in[m,M)$ and
\[
\max_P\Psi=3F(m)+F(U).
\]
Equality holds if and only if $(u_1,u_2,u_3,u_4)$ is a permutation of $(U,m,m,m)$.
\end{lemma}

\begin{proof}
Since $3m+M>2\pi$, we have $U<M$. Also $m\le \pi/2$ by Lemma~\ref{lem:a-range}, hence
\[
U=2\pi-3m\ge 2\pi-3\cdot\frac{\pi}{2}=\frac{\pi}{2}\ge m.
\]
Thus $U\in[m,M)$.

Consider the ordered feasible set
\[
P^\downarrow:=\{(v_1,v_2,v_3,v_4)\in[m,M)^4:\ v_1\ge v_2\ge v_3\ge v_4,\ v_1+v_2+v_3+v_4=2\pi\}.
\]
Since $\Psi$ is symmetric, it suffices to maximize $\Psi$ on $P^\downarrow$. Every point of $P^\downarrow$ lies in the compact box $[m,U]^4$: indeed, if $(v_1,v_2,v_3,v_4)\in P^\downarrow$, then
\[
v_1=2\pi-(v_2+v_3+v_4)\le 2\pi-3m=U,
\]
and therefore $v_k\le U$ for all $k$. Hence $\Psi$ attains its maximum on $P^\downarrow$.

Now let $w=(U,m,m,m)$. We claim that $w \in P^\downarrow$ majorizes every $v=(v_1,v_2,v_3,v_4)\in P^\downarrow$. Indeed,
\[
v_1\le U,
\qquad
v_1+v_2\le 2\pi-2m=U+m,
\qquad
v_1+v_2+v_3\le 2\pi-m=U+2m,
\]
and both vectors have total sum $2\pi$. By Karamata's inequality and strict convexity of $F$,
\[
\Psi(v_1,v_2,v_3,v_4)\le F(U)+3F(m),
\]
with equality if and only if $v$ is a permutation of $(U,m,m,m)$.
\end{proof}


\begin{lemma}\label{lem:GimpliesTight}
Assume $b>0$ and $0\le a\le 1$. If $G(a,b)\le 0$, then $3m+M>2\pi$.
\end{lemma}

\begin{proof}
The argument is presented stepwise.

\medskip
\noindent\emph{Step 1: rewrite $G$ as a quadratic function in $s=b^2$.}
Let $s=b^2$, then
\[
G(a,b)=s^2+s(2a^2+2a-1)+(a^2+a)^2+2a^2.
\]
with discriminant equal to
\[
\Delta=-(2a+1)(6a-1).
\]

Hence $G(a,b)\le 0$ is possible only when $0\le a \le  1/6$.

\medskip
\noindent\emph{Step 2: the case $a=0$.}
If $a=0$, then $\lambda=ib$, $m=Arg(\lambda)=\pi/2$ and $M=Arg(\lambda-1)=\pi-\arctan(b)$, so
\[
3m+M=\frac{5\pi}{2}-\arctan(b)>2\pi.
\]
Thus the claim holds.

\medskip
\noindent\emph{Step 3: show $b^2>3a^2$.}
Assume now $0<a<1/6$. Let $s_-(a)$ denote the smaller root of the quadratic polynomial in $s$ corresponding to $G(a,b)$. Since $G(a,b)\le 0$, we have $s=b^2\ge s_-(a)$, where
\[
s_-(a)=\frac{1-2a-2a^2-\sqrt{(2a+1)(1-6a)}}{2}.
\]

Additionally, for every $a\in(0,1/6)$, we have
\[
s_-(a)>3a^2  \iff  32a^3+64a^4>0.
\]
Since the last inequality holds, we become $b^2>3a^2$.


\medskip
\noindent\emph{Step 4: convert $3m+M>2\pi$ into an algebraic inequality.}
Since $a<1$ and $b>0$, we have
\[
M=\pi-\arctan\!\Big(\frac{b}{1-a}\Big).
\]
Thus $3m+M>2\pi$ is equivalent to
\[
\tan(3m)>\frac{b}{1-a}.
\]
Using $\tan m=b/a$ and the triple-angle formula,
\[
\tan(3m)=\frac{b(b^2-3a^2)}{a(3b^2-a^2)}.
\]
Because $b^2>3a^2$, all involved factors are positive and by introducing
\begin{equation}\label{Eq:N(a,b)}
N(a,b):=4a^3-3a^2-4ab^2+b^2
\end{equation}
we get
\[3m+M>2\pi \iff
\tan(3m)>\frac{b}{1-a} \iff N(a,b)>0.
\]

\medskip
\noindent\emph{Step 5: compare roots.}
For fixed $a\in(0,1/6)$, the expression $N(a,b)$ is affine and strictly increasing in $s=b^2$ because $1-4a>0$. Its unique zero occurs at
\[
s_0(a)=\frac{a^2(3-4a)}{1-4a}.
\]
A direct computation shows that
\[
s_-(a)>s_0(a) \iff (1-6a+16a^3)^2>(1-4a)^2(2a+1)(1-6a),
\]
whose left-hand side exceeds the right-hand side by $256a^6$. Consequently, $s\ge s_-(a)$ implies $s>s_0(a)$ and hence $N(a,b)>0$. Therefore $3m+M>2\pi$.

\medskip
\noindent\emph{Step 6: the case $a=1/6$.}
In this case there exists a point for which \(G(a,b)=0\) holds, namely \(a=1/6\) and \(b=\sqrt{11}/6\). One can easily verify that \(b^2>3a^2\) and \(N(a,b)>0\). Through step 4 it also follows that \(3m +M > 2 \pi\).
\end{proof}

\begin{proposition}\label{prop:G-condition}
If $\lambda=a+ib$ is a nonreal eigenvalue of some matrix $A(\alpha_1,\alpha_2,\alpha_3, \alpha_4)$, then
\[
G(a,b)\ge 0.
\]
\end{proposition}

\begin{proof}
It is sufficient to treat the case $b>0$. Suppose $\lambda \in \sigma(A(\alpha_1,\alpha_2,\alpha_3, \alpha_4))$. Then $0\in\Psi(P)$ by Proposition~\ref{prop:criterion}, so in particular $\sup_P\Psi\ge 0$.

In the case $3m+M\le 2\pi$, Lemma~\ref{lem:GimpliesTight} implies, by contraposition, that $G(a,b)>0$.

In the case $3m+M>2\pi$, we have by Lemma~\ref{lem:tight},
\[
\max_P\Psi=3F(m)+F(U),\qquad U=2\pi-3m.
\]
We now compute this explicitly. Since $t(m)=1$, we have
\[
F(m)=\log|\lambda|.
\]
Also, by (\ref{F(u) goniometrical form}), we can express $F(U)$ as
\[
F(U)=\log(b\csc U)-\log(b\cotan U+1-a).
\]
Because $U=2\pi-3m$, $\csc m = \frac{|\lambda|}{b}$ and $\cotan m = \frac{a}{b}$, triple-angle identities yield
\[
\csc U=-\frac{|\lambda|^3}{b(3a^2-b^2)},
\qquad
\cotan U=-\frac{a(a^2-3b^2)}{b(3a^2-b^2)}.
\]
Hence
\[
\max_P\Psi=\log\!\Big(\frac{|\lambda|^6}{N(a,b)}\Big),
\]
where $N(a,b)$ as introduced in (\ref{Eq:N(a,b)}), and we get
\[
\max_P\Psi\ge 0
\iff |\lambda|^6\ge N(a,b).
\]
Moreover,
\begin{equation}\label{Eq:factorization N(a,b) and G(a,b)}
|\lambda|^6 - N(a,b) = \left((a - 1)^2 + b^2\right) G(a,b)  
\end{equation}
where $(a - 1)^2 + b^2 >0$ since $b \neq 0$.
Hence, 
\[
\max_P\Psi\ge 0
\iff |\lambda|^6\ge N(a,b)
\iff G(a,b)\ge 0.
\]
so that we can conclude that $G(a,b)\ge 0$. Replacing $b$ by $|b|$ completes the proof.
\end{proof}


\section{Boundary of the spectral region}\label{sec:boundary}
In this section we prove that the boundary of the spectral region \(\mathcal{R}\) consist of CR and CL as presented in Figure \ref{spectral_region}. For this, we have to demonstrate that the points of CL and CR are indeed eigenvalues of realizing $4$-cycle stochastic matrices.
\begin{proposition}[Right boundary segment]\label{prop:CR}
For each $x\in[0,1]$, the matrix obtained by setting
\[
\alpha_1=\alpha_2=\alpha_3=\alpha_4=1-x
\]
has the nonreal eigenvalues $1-x\pm ix$. In particular, the upper segment
\[
CR:\qquad \lambda=1-x+ix,\qquad x\in[0,1],
\]
consists of eigenvalues of $4$-cycle stochastic matrices.
\end{proposition}

\begin{proof}
As it is known that \(\lambda = 1\) is always an eigenvalue of a stochastic matrix, consider from now on \(x\in(0,1]\). If $\alpha_1=\alpha_2=\alpha_3=\alpha_4=1-x$, then $t_1=t_2=t_3=t_4=x$, and \eqref{eq:zproduct-new} becomes
\[
(z+x)^4=x^4.
\]
Hence $z+x=xe^{ik\pi/2}$ for $k=0,1,2,3$. The nonreal solutions correspond to $k=1,3$, namely
\[
z+x=\pm ix.
\]
Since $z=\lambda-1$, this gives $\lambda=1-x\pm ix$.
\end{proof}

For the left boundary curve $G(a,b)=0$, realizing matrices are presented in \cite{ran2024nonnegative}. For completeness, we include a formal proof in Proposition \ref{prop:CL}.

\begin{definition}\label{RealizingMatrixLeftBound}
For $\alpha\in[0,1)$, define
\[
A_L(\alpha)=
\begin{pmatrix}
\alpha & 1-\alpha & 0 & 0\\
0 & 0 & 1 & 0\\
0 & 0 & 0 & 1\\
1 & 0 & 0 & 0
\end{pmatrix}.
\]
\end{definition}

\begin{proposition}[Left boundary curve]\label{prop:CL}
Let
\(
H:=\{z\in\mathbb C:\operatorname{Im}z>0\}
\)
and define
\[
C_L^+:=\{a+ib:b>0,\ a\ge 0,\ a+b\le 1,\ G(a,b)=0\}.
\]
For $\alpha\in[0,1)$, let $A_L(\alpha)$ be as in Definition \ref{RealizingMatrixLeftBound}. Then
\[
\bigcup_{\alpha\in[0,1)}
\left(\sigma(A_L(\alpha))\cap H\right)=C_L^+.
\]
Equivalently, the upper-half-plane nonreal eigenvalues of the matrices
$A_L(\alpha)$, $\alpha\in[0,1)$, are exactly the points of the left
boundary portion $C_L^+$.

More explicitly, if $\lambda=a+ib\in C_L^+$, then
$\lambda\in\sigma(A_L(\alpha))$ for
\[
\alpha
=1-\frac{b^2-3a^2}{2a+1}
=\frac{3a^2+2a+1-b^2}{2a+1}.
\]
The point $i$ is realized for $\alpha=0$. The point $0$ is not realized
for any $\alpha\in[0,1)$, but it is the limit of the upper-half-plane
nonreal eigenvalue as $\alpha\uparrow 1$.
\end{proposition}

\begin{proof}
Put
\(
t:=1-\alpha.
\)
Then $\alpha\in[0,1)$ is equivalent to $t\in(0,1]$. The characteristic
polynomial of $A_L(\alpha)$ is
\[
p_\alpha(\lambda)=\lambda^4-\alpha\lambda^3+\alpha-1.
\]
In terms of $t$, this factors as
\[
p_\alpha(\lambda)
=(\lambda-1)\bigl(\lambda^3+t(\lambda^2+\lambda+1)\bigr).
\]
Set
\[
q_t(\lambda):=\lambda^3+t(\lambda^2+\lambda+1).
\]
Since every $\lambda\in H$ satisfies $\lambda\ne1$, we have
\[
\lambda\in\sigma(A_L(\alpha))\cap H
\iff q_t(\lambda)=0
\]
for some $t\in(0,1]$.

Write $\lambda=a+ib$, $b>0$, $s:=b^2$. A direct expansion gives
\[
q_t(a+ib)
=
\left[
a^3-3as+t(a^2-s+a+1)
\right]
+ib\left[
3a^2-s+t(2a+1)
\right].
\]
Thus $q_t(a+ib)=0$ is equivalent to \(\operatorname{Im}(q_t(a+bi))=0\) and \(\operatorname{Re}(q_t(a+bi))=0\):
\begin{align}
3a^2-s+t(2a+1)&=0, \label{eq:imag}\\
a^3-3as+t(a^2-s+a+1)&=0. \label{eq:real}
\end{align}
We first prove
\[
\bigcup_{\alpha\in[0,1)}
\left(\sigma(A_L(\alpha))\cap H\right)\subseteq C_L^+.
\]
Assume that $q_t(\lambda)=0$ for some $t\in(0,1]$, with
$\lambda=a+ib\in H$. The real polynomial
\[
q_t(x)=x^3+tx^2+tx+t
\]
has derivative
\[
q_t'(x)=3x^2+2tx+t.
\]
The discriminant of $q_t'$ is
\[
4t(t-3)<0,
\]
and the leading coefficient is positive. Hence $q_t'(x)>0$ for every
real $x$. Therefore $q_t$ is strictly increasing on $\mathbb R$ and has
exactly one real root. Denote this root by $r$.

Now
\[
q_t(-1)=t-1\le0,
\qquad
q_t(-t)=t(1-t)\ge0.
\]
Since $q_t$ is strictly increasing, we know
\(
r\in[-1,-t].
\)
Set
\(
u:=-r.
\)
Then
\(
u\in[t,1].
\)
The three roots of $q_t$ are $r,\lambda,\overline{\lambda}$. By Vieta's
formulas,
\[
r+\lambda+\overline{\lambda}=-t,
\]
so
\[
r+2a=-t.
\]
Therefore
\[
a=\frac{-t-r}{2}=\frac{u-t}{2}\ge0.
\]
Also, Vieta's formula for the product of the roots gives
\[
r\lambda\overline{\lambda}=-t,
\]
hence
\[
r(a^2+b^2)=-t.
\]
Since $r=-u$,
\[
a^2+b^2=\frac{t}{u}.
\]
We next prove $a+b\le1$. Using
\[
a=\frac{u-t}{2},
\qquad
a^2+b^2=\frac{t}{u},
\]
one obtains
\[
(1-a)^2-b^2
=(1-a)^2+a^2-\frac{t}{u}
=
\frac{(t-u)(tu-u^2+2u-2)}{2u}.
\]
Since $u\in[t,1]$, we have $t-u\le0$. Define
\[
\phi(u):=tu-u^2+2u-2.
\]
Then
\[
\phi'(u)=t+2-2u\ge t>0
\]
for $u\in[t,1]$, so $\phi$ is increasing on $[t,1]$. Hence
\(
\phi(u)\le\phi(1)=t-1\le0.
\)
Thus both factors in the numerator are nonpositive, and $u>0$, so
\[
(1-a)^2-b^2\ge0.
\]
Moreover,
\[
a=\frac{u-t}{2}\le\frac{1-t}{2}<1,
\]
so $1-a>0$. Since $b>0$, the inequality above implies
\[
b\le1-a,
\]
and therefore
\[
a+b\le1.
\]
It remains to prove $G(a,b)=0$. Since $a\ge0$, we have $2a+1>0$. From
\eqref{eq:imag},
\[
t=\frac{s-3a^2}{2a+1}.
\]
Substituting this value of $t$ into the real equation \eqref{eq:real}
gives
\[
0
=
(2a+1)\left[
a^3-3as+
\frac{s-3a^2}{2a+1}(a^2-s+a+1)
\right].
\]
Expanding the right-hand side yields
\[
(2a+1)\left[
a^3-3as+
\frac{s-3a^2}{2a+1}(a^2-s+a+1)
\right]
=
-\left[(s+a^2+a)^2+2a^2-s\right].
\]
Since $s=b^2$, this is precisely $-G(a,b)$. Therefore
\[
G(a,b)=0.
\]
We have shown that
\[
b>0,\qquad a\ge0,\qquad a+b\le1,\qquad G(a,b)=0,
\]
so $\lambda\in C_L^+$.

We now prove the reverse inclusion. Let
\[
\lambda=a+ib\in C_L^+.
\]
Thus
\[
b>0,\qquad a\ge0,\qquad a+b\le1,\qquad G(a,b)=0.
\]
Again set $s=b^2$, and define
\[
t:=\frac{s-3a^2}{2a+1}.
\]
We claim first that $t\in(0,1]$.

If $a=0$, then
\[
G(0,b)=b^4-b^2=b^2(b^2-1).
\]
Since $b>0$ and $a+b=b\le1$, this forces $b=1$. Hence $s=1$ and
\[
t=\frac{s}{1}=1.
\]

Assume now that $a>0$. We prove that $s>3a^2$. Suppose, to the contrary,
that $s\le3a^2$. Then
\[
s=ca^2
\]
for some $c\in[0,3]$. Substitution into $G$ gives
\[
G(a,b)
=
(ca^2+a^2+a)^2+2a^2-ca^2
=
a^2\left[(1+(c+1)a)^2+2-c\right].
\]
The expression within the bracket is strictly positive for every $c\in[0,3]$. Indeed, if
$c<3$, then
\[
(1+(c+1)a)^2+2-c>3-c>0,
\]
while for $c=3$,
\[
(1+4a)^2-1>0.
\]
Thus $G(a,b)>0$, contradicting $G(a,b)=0$. Therefore
\[
s>3a^2.
\]
Since $2a+1>0$, this gives $t>0$.
It remains to show $t\le1$. This is equivalent to
\[
s-3a^2\le2a+1.
\]
Since $a+b\le1$, we have $b\le1-a$, and hence
\[
s=b^2\le(1-a)^2.
\]
Therefore
\[
s-3a^2
\le
(1-a)^2-3a^2
=
1-2a-2a^2
\le
1+2a.
\]
Thus $t\le1$. Hence $t\in(0,1]$.
Define
\[
\alpha:=1-t.
\]
Then $\alpha\in[0,1)$. We now show that
$\lambda\in\sigma(A_L(\alpha))$. By the definition of $t$,
\[
3a^2-s+t(2a+1)=0,
\]
so the imaginary part of $q_t(a+ib)$ vanishes. For the real part, using
the same substitution as above,
\[
(2a+1)\operatorname{Re}(q_t(a+ib))=-G(a,b).
\]
Since $G(a,b)=0$, the real part also vanishes. Hence
\[
q_t(a+ib)=0.
\]
Consequently,
\[
p_\alpha(\lambda)=(\lambda-1)q_t(\lambda)=0.
\]
Thus
\[
\lambda\in\sigma(A_L(\alpha)).
\]
This proves
\[
C_L^+\subseteq
\bigcup_{\alpha\in[0,1)}
\left(\sigma(A_L(\alpha))\cap H\right).
\]
Combining the two inclusions gives the asserted equality.

Finally, we verify the endpoint statement. If $\alpha=0$, then $t=1$, and
\[
q_1(\lambda)=\lambda^3+\lambda^2+\lambda+1
=(\lambda+1)(\lambda^2+1).
\]
Thus the upper-half-plane nonreal root is $\lambda=i$.

Now let $\alpha\uparrow1$, equivalently $t\downarrow0$. For each
$t\in(0,1]$, the polynomial $q_t$ has exactly one real root and one
conjugate pair of nonreal roots, as shown above. Let
\[
\lambda_t=a_t+ib_t,\qquad b_t>0,
\]
be its upper-half-plane nonreal root, and let $r_t$ be its real root. We
already proved that
\[
r_t\in[-1,-t].
\]
If $t\downarrow0$, then every convergent subsequence of $r_{t}$ has a
limit $r\in[-1,0]$. Taking the limit in
\[
r_{t}^3+t(r_{t}^2+r_{t}+1)=0
\]
gives
\[
r^3=0,
\]
so $r=0$. Therefore
\[
r_t\to0.
\]
Writing $u_t:=-r_t$, Vieta's formula gives
\[
a_t=\frac{u_t-t}{2}\to0.
\]
Equation (\ref{eq:imag}) for $q_t(\lambda_t)=0$ gives
\[
b_t^2=3a_t^2+t(2a_t+1)\to0.
\]
Hence
\[
\lambda_t\to0.
\]
On the other hand,
\[
p_\alpha(0)=\alpha-1\ne0
\]
for every $\alpha\in[0,1)$. Thus $0$ is not realized for any admissible
$\alpha$, but it is a limit point as $\alpha\uparrow1$.
\end{proof}

\section{Every interior point of $\mathcal{R}$ is realized} \label{sec:converse}
In Lemma \ref{lem:starconvex} we prove that the set of eigenvalues of $4$-cycle stochastic matrices is star-convex with respect to $1$. This star-convexity property is then helpful in writing interior points in terms of boundary points in Theorem \ref{thm:interior}.
\begin{lemma}[Star-convexity with respect to \(1\)]\label{lem:starconvex}
Assume that \(\mu\in\sigma(A(\alpha_1,\alpha_2,\alpha_3,\alpha_4))\), where
\(\alpha_1,\alpha_2,\alpha_3,\alpha_4\in[0,1)\). Then for every \(l\in(0,1]\),
\[
\lambda_l:=(1-l)+l\mu
\]
is an eigenvalue of a \(4\)-cycle stochastic matrix with parameters in
\([0,1)\).
\end{lemma}

\begin{proof}
Define
\[
\alpha_{k,l}:=(1-l)+ l \alpha_k,\qquad k=1,2,3,4.
\]
Then
\[
1-\alpha_{k,l}=l(1-\alpha_k),
\]
and therefore
\[
A(\alpha_{1,l},\alpha_{2,l},\alpha_{3,l},\alpha_{4,l})
=(1-l)I+lA(\alpha_1,\alpha_2,\alpha_3,\alpha_4).
\]
If \(Av=\mu v\), then
\[
\big((1-l)I+lA\big)v=\big((1-l)+l\mu\big)v.
\]
Since \(l>0\) and \(\alpha_k<1\), we have \(\alpha_{k,l}<1\), and also \(\alpha_{k,l}\ge 0\).Thus \(\lambda_l=(1-l)+l\mu\) is an eigenvalue of the $4$-cycle stochastic matrix $A(\alpha_{1,l},\alpha_{2,l},\alpha_{3,l},\alpha_{4,l})$. 
\end{proof}

\begin{lemma}[Geometric intersection with the left boundary]\label{Geometric intersection with the left boundary}
Let
\(
\lambda = a + ib
\)
satisfy
\[
b > 0, \qquad 0 \le a < 1, \qquad a+b < 1, \qquad G(a,b) > 0.
\]
Then there exists a point
\(
\mu
\)
on the upper-half-plane branch of the algebraic curve
\[
G(a,b)=0
\]
such that
\[
\lambda = (1-\ell)+\ell\mu
\]
for some \(\ell \in (0,1)\).

\begin{proof}
Consider the ray starting at \(1\) and passing through \(\lambda\):
\[
\rho(t)
=
1+t(\lambda-1)
=
1-t(1-a)+itb,
\qquad t\ge0.
\]
Clearly,
\(
\rho(1)=\lambda.
\) Let
\(
t_0=\frac{1}{1-a}>1,
\)
then
\(
\operatorname{Re}(\rho(t_0))=0,
\)
and
\(
\rho(t_0)=i\frac{b}{1-a}.
\)
Since
\(
a+b<1,
\)
we obtain
\(
0<\frac{b}{1-a}<1.
\) Therefore,
\[
G\left(0,\frac{b}{1-a}\right)
=
\left(\frac{b}{1-a}\right)^4
-
\left(\frac{b}{1-a}\right)^2
<0.
\]
On the other hand,
\[
G(\rho(1))=G(a,b)>0.
\]
By continuity, there exists
\(
t^*\in(1,t_0)
\)
such that
\[
G(\rho(t^*))=0.
\]
Define
\(
\mu:=\rho(t^*)
\) and since \(t^*>1\), we may write
\[
\mu
=
1+t^*(\lambda-1),
\]
which rearranges to
\[
\lambda
=
\left(1-\frac{1}{t^*}\right)
+\frac{1}{t^*}\mu.
\]
Setting
\(
\ell:=\frac{1}{t^*}\in(0,1),
\) we obtain
\[
\lambda=(1-\ell)+\ell\mu.
\]
Finally, for every \(t\in[1,t_0]\),
\[
\operatorname{Re}(\rho(t))
+
\operatorname{Im}(\rho(t))
=
1-t(1-a-b)
<1.
\]
Hence \(\mu\) lies on the relevant left boundary branch \(G(a,b)=0\) in the upper half-plane.
\end{proof}
\end{lemma}

\begin{theorem}[Interior realizability]\label{thm:interior}
Let \(\lambda=a+ib\) satisfy
\[
b>0,\qquad 0\le a\le 1,\qquad a+b<1,\qquad G(a,b)>0.
\]
Then there exist \(\alpha_1,\alpha_2,\alpha_3,\alpha_4\in[0,1)\) such that
\[
\lambda\in\sigma(A(\alpha_1,\alpha_2,\alpha_3,\alpha_4)).
\]
\end{theorem}
\begin{proof}
Let
\(
\lambda = a+ib
\)
satisfy
\[
b>0, \qquad 0\le a<1, \qquad a+b<1, \qquad G(a,b)>0.
\]
By Lemma \ref{Geometric intersection with the left boundary}, there exists a point \(\mu\) on the upper-half-plane branch of
\(
G(a,b)=0
\)
and a scalar \(\ell\in(0,1)\) such that
\[
\lambda=(1-\ell)+\ell\mu.
\]
By Proposition \ref{prop:CL}, the point \(\mu\) is realized as an eigenvalue of a matrix
\(
A_L(\alpha)
\)
with \(\alpha\in[0,1)\). Applying Lemma \ref{lem:starconvex} therefore shows that \(\lambda\) is realized by
\[
(1-\ell)I+\ell A_L(\alpha),
\]
which is again a 4-cycle stochastic matrix with parameters in \([0,1)\).
\end{proof}

\section{Conclusions and further research questions}

The proof of Theorem \ref{thm:main} establishes the complete description of the nonreal spectral region of the 4-cycle stochastic matrices. The necessity part follows from Lemma \ref{lem:a-range}, Proposition \ref{prop:right-boundary-necessary}, and Proposition \ref{prop:G-condition}, which show that every nonreal eigenvalue must satisfy the inequalities \(0 \leq a < 1\), \(a+b\leq 1\), and \(G(a,b)\geq 0\). Conversely, Proposition \ref{prop:CR}, Proposition~\ref{prop:CL}, and Theorem \ref{thm:interior} show that the right boundary segment, the left boundary curve, and every point in the strict interior of the nonreal eigenvalue region \(\mathcal{R}\) are realized. Since all matrices under consideration are real, the lower half-plane follows by conjugation symmetry. Combining these results yields Theorem \ref{thm:main}. Together with Theorem \ref{prop:real-eigenvalues}, which characterizes the real spectral region as \([-1,1]\), this proves the Main Theorem \ref{MainTheorem}.

The spectral regions of \(n \times n\) stochastic matrices in full generality are described by the Karpelevich theorem \cite{Karpelevich}. For subclasses of stochastic matrices, however, such a complete description is far from known. The present paper settles this problem for the 4-cycle stochastic matrices, but the corresponding question for \(n\)-cycle stochastic matrices remains open for \(n \geq 5\). Extending the ideas developed here to arbitrary \(n \geq 5\) would require a much deeper understanding of the geometry of the associated argument parametrizations, convexity constraints, and boundary realizations. In particular, disseminating the methods of the present paper from \(n=4\) to general \(n \geq 5\) is a substantial and highly nontrivial challenge. Developing such a general theory for spectral regions of \(n\)-cycle stochastic matrices is therefore a natural and important direction for future research.

\section{Declaration of generative AI and AI-assisted technologies in the manuscript preparation process}
During the preparation of this work the author(s) used ChatGPT-5.2 (Thinking) as an AI-assisted research copilot for the generation, critique, and repair of candidate mathematical arguments, as well as for exploring proof strategies and alternative approaches to the studied conjecture. After using this tool, the author(s) reviewed and edited the content as needed and take(s) full responsibility for the content of the published article.

\bibliographystyle{plain}
\bibliography{references}







\end{document}